\documentclass[10pt,a4paper]{amsart}
\usepackage{amssymb,amsmath}
\newtheorem{thm}{Theorem}[section]

\newtheorem{lem}[thm]{Lemma}

\newtheorem{rek}[thm]{Remark}

\newcommand\be{\begin{equation}}
\newcommand\ee{\end{equation}}
\newcommand\bea{\begin{eqnarray}}
\newcommand\eea{\end{eqnarray}}
\newcommand\bi{\begin{itemize}}
\newcommand\ei{\end{itemize}}
\newcommand\ben{\begin{enumerate}}
\newcommand\een{\end{enumerate}}
\newcommand\bc{\begin{center}}
\newcommand\ec{\end{center}}
\newcommand\ba{\begin{array}}
\newcommand\ea{\end{array}}

\newcommand{\R}{\ensuremath{\mathbb{R}}}

\newcommand{\fof}{\frac{1}{4}}  

\newcommand{\gl}{\lambda}
\newcommand{\gd}{\delta}
\newcommand{\gD}{\Delta}
\newcommand{\gO}{\Omega}
\newcommand{\go}{\omega}
\newcommand{\dgO}{\partial\Omega}
\newcommand{\gs}{\sigma}
\newcommand{\gG}{\Gamma}
\newcommand{\gGt}{\tilde{\Gamma}}

\begin{document}

\begin{center}
{\large\textbf{A RECONSTRUCTION PROCEDURE FOR THERMOACOUSTIC TOMOGRAPHY IN THE CASE OF LIMITED BOUNDARY DATA}\\
\bigskip
\medskip
Dustin \textsc{Steinhauer}\\
\bigskip
University of California, Los Angeles}
\end{center}
\bigskip
\bigskip

\textbf{Abstract.} We derive an explicit method for reconstructing singularities of the initial data in a thermoacoustic tomography problem, in the case of variable sound speed and limited boundary data.  In order to obtain this explicit formula we assume the metric induced by the sound speed does not have conjugate points inside the region to be observed.



\section{Reconstruction of the Wavefront Set Given Measurements on a Hyperplane}

Let $\gO$ be a bounded region in $\R^n$ such that $x_n<0$ for $x\in\gO$.  Let $H$ be the hyperplane $\{x_n=0\}$.  We study the Cauchy problem

\bea 
\label{cauchy}
Pu=u_{tt}-c^2(x)\gD u&=&0 \nonumber\\
u(x,0)&=&f(x)\nonumber\\
u_t(x,0)&=&0 \;,\eea
where $c$ is a smooth function which satisfies $\frac{1}{M}<c(x)<M$ for some number $M>1$.  This is the Cauchy problem commonly studied in mathematical investigations of thermoacoustic tomography and photoacoustic tomography (see \cite{KuKy} for a survey of results in this area).

The differential operator $P$ has symbol

\be p(x,\xi,\tau)=\tau^2-c^2(x)|\xi|^2 \;.\ee
$c(x)$ induces a Riemannian metric on $\R^n$ by setting $g_{ij}(x)=c^{-2}(x)\gd_{ij}$.  We assume $c(x)$ equals 1 outside $\gO$, so this metric equals the Euclidean metric outside $\gO$.  Our goal is to reconstruct singularities of the initial data $f$ given information about $u$ on $H\times\R_t$.  We therefore allow $f$ to be any distribution in $\mathcal{E}'(\gO)$.

The solution to (\ref{cauchy}) can be expressed as a sum of two Fourier integral operators applied to $f$, as in \cite{Du}:

\be u=(E_++E_-)(f)\;. \ee

$E_+$ and $E_-$ are of order $-\fof$ with nonvanishing principal symbols, and have canonical relations $\widetilde{\mathcal{H}_\pm}$ which are, roughly speaking, the union of bicharacteristic strips.  A bicharacteristic strip $\gamma(t)=(x(t),t,\xi(t),\tau)$ with initial condition $(y,\eta)$ is a smooth curve in $T^*(\R^n\times\R)$ defined by

\bea d_t(x(t),\xi(t))&=&\frac{1}{2\tau}(\nabla_\xi p,-\nabla_x p)=(2c^2(x)\xi,-\nabla_x(c^2(x))|\xi|^2)\nonumber\\
(x(0),\xi(0))&=&(y,\eta)\;.\eea
Since $p$ is independent of $t$, $\tau$ is constant and is determined by the condition $p(\gamma)=0$.  Thus given any initial condition $(y,\eta)$ there are two bicharacteristics with that initial condition, which we will denote by $\gamma^+_{(y,\eta)}(t),\gamma^-_{(y,\eta)}(t)$ where $\tau>0$ (resp. $<0$) for $\gamma^+$ (resp. $\gamma^-$).  We then have

\be\label{bichar}
\widetilde{\mathcal{H}_\pm}=\{(y,\eta,x,t,\xi,\tau)\:|\:(x,t,\xi,\tau)=\gamma^\pm_{(y,\eta)}(t)\}\;.\ee
Since $c(x)=1$ outside $\gO$, each bicharacteristic intersects $H\times\R$ at most once for $t>0$, and at most once for $t<0$.

Our goal is to reconstruct singularities of $f$ from knowledge of the solution $u$ on a portion $\dgO\times\R$.  However, we first consider a local problem in which we assume $u$ is known on $H\times\R$.

There exists $T>0$ such that the eikonal equation

\bea \label{ike} p(x,\nabla_x\phi,d_t\phi)=|d_t\phi|^2-c^2(x)|\nabla_x\phi|^2&=&0\nonumber\\
\phi(x,0,\xi,\tau)&=&\langle x,\xi\rangle \eea
has two smooth solutions $\phi_\pm$ on $\R^n_x\times[-T,T]_t$.  We also have $\nabla_x\phi_\pm\neq0$ on $\R^n\times[-T,T]$, so since $\phi_\pm$ are homogeneous of degree 1 in $\xi$, we have for any $K\subset\R^n$ compact,

\be \label{nonv} |\nabla_x\phi_\pm(x,t,\xi)|\geq C|\xi| \ee
on $K\times[-T,T]$.  In fact, since $c(x)=1$ outside $\gO$, we can take $C$ uniform on $ \R^n\times[-T,T]$.


Define two Fourier integral operators $S_\pm$ of zeroth order by

\bea \label{schweeet}
(S_+ f)(x',t)&=&\frac{1}{2(2\pi)^n}\iint e^{i(\phi_+(x',0,t,\eta)-\langle y,\eta\rangle)}a_+(x',0,t,\eta)f(y)\:dy\:d\eta\nonumber\\
(S_-f)(x',t)&=&\frac{1}{2(2\pi)^n}\iint e^{i(\phi_-(x',0,t,\eta)-\langle y,\eta\rangle)}a_-(x',0,t,\eta)f(y)\:dy\:d\eta\;. \eea
The amplitudes $a_\pm=a_\pm(x,t,\eta)\in S^0_{cl}$ are constructed by geometrical optics so that the function

\be\label{ballz} u(x',t):=(S_+ +S_-)(f)(x',t) \ee
is the solution to (\ref{cauchy}) restricted to $H$, modulo smoothing operators.  This representation is valid for $|t|<T$.  We have $a_+(x,0,\eta)=a_-(x,0,\eta)=1$ and since they are constructed as solutions of transport equations, their principal symbols are nonvanishing.

$S_+$ and $S_-$ have formal adjoints acting on $v\in\mathcal{C}_0^\infty(H\times [-T,T])$ given by

\bea (S_+^*v)(y)&=&\frac{1}{2(2\pi)^n}\iiint e^{i(\langle y,\eta\rangle-\phi_+(x',0,t,\eta))}\overline{
a_+(x',0,t,\eta)}v(x',t)\:dx'\:dt\:d\eta\nonumber\\
(S_-^*v)(y)&=&\frac{1}{2(2\pi)^n}\iiint e^{i(\langle y,\eta\rangle-\phi_-(x',0,t,\eta))}\overline{
a_-(x',0,t,\eta)}v(x',t)\:dx'\:dt\:d\eta \eea
and they extend to continuous maps $\mathcal{E}'(H\times [-T,T])\rightarrow\mathcal{D}(\R^n)$.

The canonical relations of $S_\pm$, which we call $\mathcal{H}_\pm$, are defined when $|t|<T$ by taking a point $(y,\eta,x',0,t,\xi,\tau)\in\widetilde{\mathcal{H}_\pm}$ and projecting $(x',0,t,\xi,\tau)$ onto $T^*(x',t)$, a single fiber in $T^*(H\times\R)$ (see Figure 1).  It is also possible to determine a point in $\mathcal{H}_\pm$ given a point $(x',t,\xi',\tau)\in T^*(H\times[-T,T])\backslash 0$ as follows: there is a unique number $\xi_n$ such that $p(x',0,\xi',\xi_n,\tau)=0$ and the bicharacteristic of $p$ through $(x',0,t,\xi',\xi_n,\tau)$ travels to the left (decreasing $x_n$) as $t$ goes towards 0.  Since $\gO$ lies to the left of $H$, this is the only choice of $\xi_n$ such that the resulting bicharacteristic might lie in $\gO$ when $t=0$.  Continue this bicharacteristic through $(x',0,t,\xi',\xi_n,\tau)$ back to $\{t=0\}$ and if it arrives in $\gO$, project onto $T^*(\gO)$ to obtain a point $(y,\eta)$.  We can thus write

\be \label{can1}
\mathcal{H}_\pm = \{(y,\eta,x',t,\xi',\tau)\:|\:(y,\eta,x',0,t,\xi',\xi_n,\tau)
\in\widetilde{\mathcal{H}_\pm}\} \;.\ee
From the local representation in (\ref{ballz}) we also have

\bea\label{can2}
\mathcal{H}_\pm &=& \{(\nabla_\eta\phi_\pm(x',0,t,\eta),\eta,x',t,\nabla_{(x',t)}\phi_\pm(x',0,t,\eta))\nonumber\\
&|&(x',t,\eta)\in\R^{n-1}_x\times[-T,T]_t\times\R^n_\eta\} \eea
We note the following consequence of the preceding:

\begin{lem}
\label{endpoint} Suppose

\be\label{endpt}\nabla_{(x',t)}\phi_\pm(x',0,t,\eta)=\nabla_{(x',t)}\phi_\pm(x',0,t,\zeta)\;.\ee
Then $\eta=\zeta$. \end{lem}
\begin{proof} Let $\xi_n<0$ be the unique number satisfying

\be p(x',0,\nabla_{x'}\phi_\pm(x',0,t,\eta),\xi_n,d_t\phi_\pm(x',0,t,\eta))=0\;. \nonumber\ee
The existence and uniqueness of $\xi_n$ means the condition in (\ref{endpt}) determines a unique bicharacteristic strip $\gamma$ whose projection from $T^*(\R^n_x\times\R_t)$ into $\R^n_x\times\R_t$ passes through $(x',0,t)$.  Continue $\gamma$ back to $\{t=0\}$.  The projection of $\gamma(0)$ from $T^*(\R^n_x\times\R_t)$ into $T^*\R^n_x$ is uniquely determined, so (\ref{bichar}) implies that $\eta=\zeta$.\end{proof}

A point $(y,\eta)\in T^*(\gO)$ will be called \textit{visible} if there exists $(x',t,\xi',\tau)$ such that either $(x',0,t,\xi',\xi_n,\tau)\in\gamma^+_{(y,\eta)}$ or $(x',0,t,\xi',\xi_n,\tau)\in\gamma^-_{(y,\eta)}$ (with $\xi_n$ determined as above).  If $|t|<T$, this is equivalent to having $(y,\eta,x',t,\xi',\tau)\in \mathcal{H}^+$ or $(y,\eta,x',t,\xi',\tau)\in \mathcal{H}^-$.  For example, if $c\equiv 1$, $(y,\eta)$ is visible if and only if $\eta_n\neq 0$.

Consider the effect of applying $S_+^*+S_-^*$ to $u$.  We must analyze the operators $S_+^*S_+$ and $S_+^*S_-$.  Neither operator is well-defined on $\mathcal{E}'(\gO)$, so cutoffs will have to be applied.  Once this problem is fixed, it will turn out that the first operator reconstructs certain visible singularities of $f$, while the second operator is smoothing.

In what follows, if $f$ and $g$ are distributions, we write $f\equiv g$ on $U\subset T^*(X)\backslash 0$ if for any pseudodifferential operator $A\in L^0$ whose essential support lies in $U$, we have $A(f-g)\in\mathcal{C}^\infty$.  Two such distributions are called \textit{microlocally equivalent modulo $\mathcal{C}^\infty$ on $U$}.

\begin{thm}
Let $\chi\in\mathcal{C}^\infty_0(H\times[-T,T])$ and let $V\subset H\times[-T,T]$ be an open set such that $\chi(x',t)=1$ for all $(x',t)\in V$  Let $U_\pm\subset T^*(\R^n)$ be open cones in $T^*(\R^n)$ such that if $(y,\eta)\in U_\pm$, the projection of the intersection of $\gamma_{(y,\eta)}^\pm$ with $T^*(H\times[-T,T])$ intersects $H\times[-T,T]$ in $V$.  Then there exist pseudodifferential operators $R_\pm\in L^0(\R^n)$ such that $(R_+S_+^*+R_-S_-^*)(S_+ +S_-)(f)\equiv f$ on $U_+\cup U_-$.\end{thm}

\begin{proof} To analyze $S_+^*\chi S_+$ we follow the arguments of \cite{GrSj}, chapter 10.  We have, in the sense of distributions,

\bea (S_+^*\chi S_+)(f)(z) = \frac{1}{4(2\pi)^{2n}} \int e^{i(\langle z,\zeta\rangle-\phi_+(x',0,t,\zeta)+\phi_+(x',0,t,\eta)-\langle y,\eta\rangle)}\nonumber\\
\overline{a_+(x',0,t,\zeta)}a_+(x',0,t,\eta)\chi(x',t) f(y)\:dy\:d\eta\:dx'\:dt\:d\zeta\;.\eea
The integral is well-defined in the sense of distributions because $\chi S_+ f\in\mathcal{E}'(H\times [-T,T])$, and $S_+^*:\mathcal{E}'(H\times [-T,T])\rightarrow\mathcal{D}'(\R_z^n)$.  (We could approximate $a_+$ by symbols in $S^{-\infty}$ to write $S_+^*\chi S_+$ as an absolutely convergent integral operator.)

We denote the phase by

\be \Phi=\langle z,\zeta\rangle-\phi_+(x',0,t,\zeta)+\phi_+(x',0,t,\eta)-\langle y,\eta\rangle\;. \ee

We break up the integral by inserting a cutoff $\psi\left(\frac{|\zeta-\eta|}{|\eta|}\right)$ and consider

\bea (S_+^*\chi S_+)(f)(z)
&=&\frac{1}{4(2\pi)^{2n}} \int e^{i\Phi}\overline{a_+(x',0,t,\zeta)}a_+(x',0,t,\eta)\nonumber\\
&\times&\psi\left(\frac{|\zeta-\eta|}{|\eta|}\right)\chi(x',t) f(y)\:dy\:d\eta\:dx'\:dt\:d\zeta\nonumber\\
&+&\frac{1}{4(2\pi)^{2n}} \int e^{i\Phi}\overline{a_+(x',0,t,\zeta)}a_+(x',0,t,\eta)\nonumber\\
&\times&\left(1-\psi\left(\frac{|\zeta-\eta|}{|\eta|}\right)\right)\chi(x',t) f(y)\:dy\:d\eta\:dx'\:dt\:d\zeta\;.\eea
By Lemma \ref{endpoint}, $\Phi$ has no stationary points in the support of the amplitude of the second integral, so we only need to consider the first.  Write $\eta=\gl\go$ where $\gl=|\eta|$ and make the change of variables $\zeta=\gl\gs$.  Then we are left with

\be \frac{\gl^n}{4(2\pi)^{2n}}\int e^{i\gl\tilde{\Phi}} \overline{a_+(x',0,t,\gl\gs)}a_+(x',0,t,\eta)\psi(|\gs-\go|)\chi(x',t) f(y)\:dy\:d\eta\:dx'\:dt\:d\gs\ee
where

\be \tilde{\Phi}=\langle z,\gs\rangle-\phi_+(x',0,t,\gs)+\phi_+(x',0,t,\go)-\langle y,\go\rangle\;. \ee

We use the method of stationary phase on the $(x',t,\gs)$-integrals.  (The integrand is compactly supported in these variables.)  By Lemma \ref{endpoint} we have the critical point $\gs=\go,\; z=\nabla_\gs\phi_+(x',0,t,\gs)|_{\gs=\go}$.  In a neighborhood of any point of $\mathcal{H}_+$, $\mathcal{H}_+$ is parametrized by the initial conditions $(y,\eta)$ of bicharacteristics, so $\mathcal{H}_+$ is a smooth manifold of dimension $2n$.  This implies that the critical point is non-degenerate (see \cite{Tre}, Proposition 8.1.2).  

\bea (S_+^*\chi S_+)(f)(z) &=& \frac{1}{2(2\pi)^n} \iint e^{i\langle z-y,\eta\rangle} b(z,\eta)\:dy\:d\eta \nonumber\\
&+& Kf(z)\;,\eea
where $K\in L^{-\infty}$.  $b\in S^0_{1,0}$ and the principal symbol is

\be\label{deeeznuuuts} b_0(z,\eta)=|a_+(x',0,t,\eta)|^2\chi(x',t)\mid_{\nabla_\eta\phi_+(x',0,t,\eta)=z}\;.\ee
In light of (\ref{can2}), we can write

\be b_0(z,\eta)=|a_+(\pi\circ\mathcal{H}_+(z,\eta),\eta)|^2\times(\chi\circ\pi\circ\mathcal{H}_+)(z,\eta) \ee
where $\pi:T^*(H\times[-S,S])\rightarrow H\times[-T,T]$ is the natural projection.

$b_0\neq0$ on $U_+$, so there exists a pseudodifferential operator $R_+'\in L^0(U_+)$ such that $(R_+'S_+\chi S_+)f\equiv f$ on $U_+$.

By a similar procedure, we can also construct $R_-'\in L^0(U_-)$ such that $(R_-'S_-\chi S_-)f\equiv f$ on $U_-$.\\

We next examine the operator $S_+^*\chi S_-$.  We have

\bea (S_+^*\chi S_-)(f)(z) = \frac{1}{4(2\pi)^{2n}} \int e^{i(\langle z,\zeta\rangle-\phi_+(x',0,t,\zeta)+\phi_-(x',0,t,\eta)-\langle y,\eta\rangle)}\nonumber\\
\overline{a_+(x',0,t,\zeta)}a_-(x',0,t,\eta)\chi(x',t) f(y)\:dy\:d\eta\:dx'\:dt\:d\zeta\;.\eea
so the Schwartz kernel of $S_+^* \chi S_-$ is

\bea K(z,y)= \int e^{i(\langle z,\zeta\rangle-\phi_+(x',t,\zeta)+\phi_-(x',t,\eta)-\langle y,\eta\rangle)} A(x',t,\eta,\zeta) d\eta\:dx'\:dt\:d\zeta \;,\eea
where

\be A(x',t,\eta,\zeta)=\frac{1}{4(2\pi)^{2n}}\overline{a_+(x',0,t,\zeta)}a_-(x',0,t,\eta)\chi(x',t)\ee
is of class $S^0_{1,0}$.

The derivative in $t$ of the phase could vanish only if $d_t\phi_+ = d_t\phi_-$.  But $\phi_+$ and $\phi_-$ are distinct solutions of the eikonal equation, so their $t$ derivatives have differing sign.  In fact,

\bea d_t\phi_+(x,t,\eta) &=& c(x)|\nabla_x\phi_+(x,t,\eta)| \nonumber\\
d_t\phi_-(x,t,\eta) &=& -c(x)|\nabla_x\phi_-(x,t,\eta)| \;.\eea
Define a differential operator

\be L = \frac{1}{i(d_t\phi_-(x',0,t,\eta)-d_t\phi_+(x',0,t,\zeta))}\:d_t \;.\ee
Then $L$ fixes the exponential portion of the integral:

\be L\left(e^{i(\langle z,\zeta\rangle-\phi_+(x',t,\zeta)+\phi_-(x',t,\eta)-\langle y,\eta\rangle)}\right)
=e^{i(\langle z,\zeta\rangle-\phi_+(x',t,\zeta)+\phi_-(x',t,\eta)-\langle y,\eta\rangle)} \;.\ee
From (\ref{nonv}) we have

\be \frac{1}{|d_t\phi_-(x',0,t,\eta)-d_t\phi_+(x',0,t,\zeta)|}\leq \frac{C}{|\eta|+|\zeta|}\;.\ee
Since $d_{tt}\phi_\pm(x,t,\xi)$ is also homogeneous in $\xi$ of degree 1, we have

\be \left|d_t\:\frac{1}{d_t\phi_-(x',0,t,\eta)-d_t\phi_+(x',0,t,\zeta)}\right|\leq\frac{C}{|\eta|+|\zeta|}\;.\ee
After repeated integration by parts using the operator $L^t$, we get

\bea K(z,y)&=&\int e^{i(\langle z,\zeta\rangle-\phi_+(x',0,t,\zeta)+\phi_-(x',0,t,\eta)-\langle y,\eta\rangle)}\left( (L^t)^k A(x',t,\eta,\zeta)\right) d\eta\:dx'\:dt\:d\zeta\nonumber\\
&\leq& \int e^{i(\langle z,\zeta\rangle-\phi_+(x',0,t,\zeta)+\phi_-(x',0,t,\eta)-\langle y,\eta\rangle)}A_{-k}(x',t,\eta,\zeta)d\eta\:dx'\:dt\:d\zeta \eea
where $A_{-k}\in S^{-k}_{1,0}$.  This can be achieved for any $k>0$.  Hence $K\in\mathcal{C}^\infty(\R^{2n})$, and $S_+^*\chi S_-\in L^{-\infty}$.  A similar argument also shows $S_-^*\chi S_+\in L^{-\infty}$.\\

We now combine $R_+'$ and $R_-'$ to form a full parametrix for $\chi(S_++S_-)$.  Let $\theta_\pm(y,\eta)\in \mathcal{C}^\infty(U_+\cup U_-)$ be homogeneous functions of degree 0 forming a partition of unity on $U_\pm$.  Let 

\be R_\pm(x,D)=\theta_\pm(x,D)R_\pm'(x,D) \;.\ee
Then by construction, on $U_+\cap U_-$ we have

\bea (R_+S_+^*+R_-S_-^*)(S_+ +S_-)(f)&\equiv&(\theta_+R_+'S_+^*S_+)(f)+\theta_-R_-'S_-^*S_-)(f)\nonumber\\
&\equiv& \theta_+f +\theta_-f\nonumber\\
&=&f\;.\eea

\end{proof}

If the metric $g_{ij}$ has no conjugate points, then we can take $T=\infty$.  As a result, for any visible singularity $(y,\eta)$ we can choose $V$ such that $(y,\eta)$ lies in $U_+\cup U_-$.

\section{Reconstruction of the Wavefront Set Given Measurements on the Boundary}

We now investigate the situation when the measurement surface $\gG$ is not a hyperplane, but instead is a relatively open subset of $\dgO$.  For simplicity we assume that $\gO$ is strictly convex, so that any bicharacteristic along which a singularity might travel intersects $\dgO$ at most once for $t>0$ and also at most once for $t<0$.  Let $\gGt$ be an open subset of $\dgO$ compactly contained in $\gG$.  We will attempt to reconstruct singularities that reach $\gGt$.

Let $W_1,\ldots W_k$ be a partition of $\gG$ such that we have boundary normal coordinates near each $W_j$.
Then, as before we define $U_{j,\pm}\subset T^*(\R^n)$ be open cones in $T^*(\R^n)$ such that if $(y,\eta)\in U_{j,\pm}$, the projection of the intersection of $\gamma_{(y,\eta)}^\pm$ with $T^*(\dgO\times\R)$ intersects $\dgO\times\R$ in $W_j\times[-T.T]$.  Let $\chi_j\in \mathcal{C}^\infty_0(T^*(\dgO\times\R))$ be cutoff functions such that $\text{supp}(\chi_j)\subset W_j\times\R$ and $\sum_j \chi_j=1$ on $\gGt\times[-T,T]$.  Our measurements correspond to knowing $u|_{\gGt\times[-T,T]}$.  Let


\be U=\bigcup_{j,\pm} U_{j,\pm}\;.\ee
We call $U$ the \textit{visibility set}, and attempt to reconstruct the singularities of $f$ on $U$.

Consider a single patch $W_j$ and let $(w',w_n)$ be boundary normal coordinates near $W_j$ such that $\gO$ corresponds to $w_n<0$.  Writing $\chi_j(S
_++S_-)$ in those coordinates we obtain data in the same form as in (\ref{schweeet}):

\bea u_j(w',t)&=&\chi_j(S_{j,+}+S_{j,-})(f)(w',t)\nonumber\\
&=&\frac{1}{2(2\pi)^n}\iint e^{i(\varphi_{j,+}(w',t,\eta)-\langle y,\eta\rangle)}\tilde{a}_{j,+}(w',t,\eta)\chi_j(w',t)f(y)\:dy\:d\eta\nonumber\\
&+&\frac{1}{2(2\pi)^n}\iint e^{i(\varphi_{j,-}(w',t,\eta)-\langle y,\eta\rangle)}\tilde{a}_{j,-}(w',t,\eta)\chi_j(w',t)f(y)\:dy\:d\eta\;, \eea
where

\bea \varphi_{j,\pm}(w',t,\eta)&=&\phi_\pm(x,t,\eta) \nonumber\\
\tilde{a}_{j,\pm}(w',t,\eta) &=& a_\pm(x,t,\eta) \eea
on $W_j\times[-T,T]$.

We apply the operator

\bea (S^*_{j,+}+S^*_{j,-})(\cdot)(z)&=&\frac{1}{2(2\pi)^n}\iiint e^{i(\langle z,\eta\rangle-\varphi_{j,+}(w',t,\eta))}\overline{
\tilde{a}_{j,+}(w',t,\eta)}(\cdot)\:dw'\:dt\:d\eta\nonumber\\
&+&\frac{1}{2(2\pi)^n}\iiint e^{i(\langle z,\eta\rangle-\varphi_{j,-}(w',t,\eta))}\overline{
\tilde{a}_{j,-}(w',t,\eta)}(\cdot)\:dw'\:dt\:d\eta \eea
to $u_j$.  As above, we obtain pseudodifferential operators of order 0, modulo smoothing operators, which are elliptic on $U_{j,+}$ and $U_{j,-}$, respectively.  Hence there exist $R_{j,\pm}\in L^0_{1,0}$ elliptic such that

\be (R_{j,+}S^*_{j,+}+R_{j,-}S^*_{j,-})\chi_j(S_{j,+}+S_{j,-})f\equiv f \ee
on $U_{j,+}\cup U_{j,-}$.

Let $\theta_j=\theta_j(y,\eta)\in\mathcal{C}^\infty(U)$ form a partition of unity on $U$ with each $\theta_j$ supported on $U_{j,+}\cup U_{j,-}$.  Then we have

\be \sum_j\theta_j(y,D_y)(R_{j,+}S^*_{j,+}+R_{j,-}S^*_{j,-})\chi_j(S_{j,+}+S_{j,-})f\equiv f \ee
on $U$.\\

In conclusion, we discuss a simple condition which guarantees that we can resonstruct the entire wavefront set of $f$.  Suppose that for each $(y,\eta)\in T^*(\gO)$, either $\gamma^+_{(y,\eta)}$ or $\gamma^-_{(y,\eta)}$ reaches $\gGt$ in a time less than $T$.  (Equivalently, we may say that either of the unit speed geodesics eminating from $(y,\eta)$ or $(y,-\eta)$ exits $\gO$ through $\gGt$ in a time less than $T$.)  Then every singularity of $f$ is visible, i.e. $T^*(\gO)\subset U$.  As a result,

\be \sum_j\theta_j(y,D_y)(R_{j,+}S^*_{j,+}+R_{j,-}S^*_{j,-})\chi_j \ee

reconstructs $f$ from $u|_{\gGt\times[-T,T]}$.  We summarize this discussion in the following theorem:

\begin{thm}

Let $f\in\mathcal{E}'(\gO)$ and let $U$ be the visible set.  Then with $\chi_j, S_{j,\pm}, R_{j,\pm}$ as above,

\be \sum_j\theta_j(y,D_y)(R_{j,+}S^*_{j,+}+R_{j,-}S^*_{j,-})\chi_j(S_{j,+}+S_{j,-})f\equiv f\ee
on $U$.

\end{thm}
\pagebreak

\begin{rek}

This work is part of the author's dissertation at UCLA.  Recently a paper \cite{StUhl} by P. Stefanov and G. Uhlmann was announced in which similar results were obtained.  Our proofs differ from those in \cite{StUhl}.
\end{rek}

\end{document}